\documentclass[11pt,amscd,amsfonts]{amsart}

\newtheorem{Theorem}{Theorem}[section]

\newtheorem{Proposition}[Theorem]{Proposition}

\newtheorem{Lemma}[Theorem]{Lemma}

\theoremstyle{remark}
\newtheorem{Example}[Theorem]{Example}

\def\re{{\mathbf {Re\,}}}

\begin{document}
\title{Pad\'e interpolation by $F$-polynomials and transfinite diameter}
\author{ Dan Coman and Evgeny A. Poletsky}
\thanks{Both authors are supported by NSF Grants.}
\address{ Department of Mathematics,  215 Carnegie Hall,
Syracuse University,  Syracuse, NY 13244-1150, USA. E-mail:
dcoman@syr.edu, eapolets@syr.edu}

\begin{abstract} We define $F$-polynomials as linear combinations of dilations by some frequencies of an entire function $F$. In this paper we use Pad\'e interpolation of holomorphic functions in the unit disk by $F$-polynomials to obtain explicitly approximating $F$-polynomials with sharp estimates on their coefficients. We show that when frequencies lie in a compact set $K\subset\mathbb C$ then optimal choices for the frequencies of interpolating polynomials are similar to Fekete points. Moreover, the minimal norms of the interpolating operators form a sequence whose rate of growth is determined by the transfinite diameter of $K$. 
\par In case of the Laplace transforms of measures on $K$, we show that the coefficients of interpolating polynomials stay bounded provided that the frequencies are Fekete points. Finally, we give a sufficient condition for measures on the unit circle which ensures  that the sums of the absolute values of the coefficients of interpolating polynomials stay bounded.
\end{abstract}

\maketitle

\section{Introduction}\label{S:intro}
\par The problem of approximation of functions on the real line by exponential polynomials has a long history. It includes, for example, the theory of Fourier series and the spectral theory of functions. In both theories it was presumed that exponents in the polynomials are purely imaginary, although there was an interest in the cases when they are arbitrary complex numbers (see \cite{HNP}).

\par In this paper we are interested in the problem of approximation of holomorphic functions by $F$-polynomials, i.e. by functions of the form
\[f(z)=\sum_{j=1}^mc_jF(q_jz),\]
where $F$ is a (transcendental) entire function and $q_j$ are distinct complex numbers. The numbers $c_j$ are called the  coefficients of $f$, while the numbers $q_j$ are the frequencies of $f$. Let ${\mathcal F}$ denote the vector space of all $F$-polynomials $f$ endowed with the norm
\[\|f\|_\infty=\max_{1\leq j\leq m}|c_j|.\]
The most interesting case is when $F(z)=e^z$. The functions $f$ are then called exponential polynomials.

\par This problem was originated in a paper of A. O. Gelfond (see \cite{Ge} and \cite[Ch. IV, Theorem 18]{Le}), but it should be noted that all results listed above were concerned with the density of $F$-polynomials in appropriate spaces and there were no constructive methods for approximations of given functions. In this paper we use Pad\'e interpolation by $F$-polynomials to obtain
explicitly approximating $F$-polynomials with estimates for their coefficients.

\par Let $F$ be an entire function and set $F_n=F^{(n)}(0)$. We assume throughout the paper that there exists a constant $\Gamma\geq1$ so that
\begin{equation}\label{e:Fn}
1\leq|F_n|\leq\Gamma,\;\;\forall\,n\geq0.
\end{equation}
We will denote by $D_R\subset{\Bbb C}$ the open disk of radius $R$ centered at $0$, by $\overline D_R$ its closure, and by $\Delta$ the closed unit disk.

\par Let $g$ be a holomorphic function in $D_R$. Since $F_n\neq0$ for all $n$, it is easy to see that, given any $m$-tuple of distinct points $Q_m=(q_{1m},\dots,q_{mm})\in{\Bbb C}^m$, there exists a unique $F$-polynomial $f$ with frequencies in $Q_m$ which interpolates $g$ to order $m$ at $0$, i.e. $f^{(j)}(0)=g^{(j)}(0)$ for $0\leq j\leq m-1$. We denote this $F$-polynomial by $T_{F,Q_m}g$. We prove in Theorem \ref{T:PadeR} that, as $m\to+\infty$ and $Q_m\in(\overline D_M)^m$ for some fixed $M>0$, the Pad\'e-Taylor interpolants $T_{F,Q_m}g$ converge to $g$ locally uniformly in $D_R$ at the same rate as the Taylor polynomials of $g$ at $0$.

\par To estimate the coefficients of interpolating $F$-polynomials, we consider for a vector $q=(q_1,\dots,q_m)$ with distinct components $q_j\in\overline D_M$ the interpolation operator 
$$T_{F,q}:(A(\Delta),\|\cdot\|_\Delta)\longrightarrow({\mathcal F},\|\cdot\|_\infty),$$ 
where $A(\Delta)$ is the space of continuous functions on $\Delta$ which are holomorphic in the open unit disk, endowed with the uniform norm $\|\cdot\|_\Delta$. We let 
$$t(q)=\|T_{F,q}\|=\sup\{\|T_{F,q}g\|_\infty:\,g\in A(\Delta),\,\|g\|_\Delta\leq1\}.$$

\par We also look at 
\[\varepsilon(q)=\min\left\{\|f\|_\Delta:\,f(z)=\sum_{j=1}^mc_jF(q_jz),\,\|f\|_\infty=1\right\},\]
which is the reciprocal of the operator norm of the restriction of $T_{F,q}$ to the space of $F$-polynomials with frequencies in $q$. When $q=\{j+k\alpha:\,0\leq j+k\leq m\}$, $\alpha\in{\Bbb R}\setminus{\Bbb Q}$, we studied $\varepsilon(q)$ in \cite{CP3} to obtain sharp estimates for the uniform norms on the bidisk of polynomials in $\mathbb C^2$ whose uniform norm on the curve $K=\{(e^z,e^{\alpha z}):\,|z|\leq1\}\subset{\Bbb C}^2$ does not exceed 1. The behavior of $\varepsilon(q)$ exposed resonance conditions for $\alpha$ and has shown that the extension of the results in \cite{CP2} will meet significant difficulties. 

\par In Theorem \ref{T:mt1} we show that
$$\frac{e^{-M}(m-1)!}{\Gamma\gamma(q)}\leq\varepsilon^{-1}(q)\leq t(q)\leq \frac{e^M(m-1)!}{\gamma(q)}\,,$$
where $q=(q_1,\dots,q_m)\in(\overline D_M)^m$, 
\[\gamma(q)=\min_{1\le i\le m}\gamma_i(q)\,\;{\rm and }\;\gamma_i(q)=\prod_{j=1,j\ne i}^m|q_i-q_j|\,.\] 

\par To get better estimates for $\|T_{F,q}\|$ we assume that for all $m$ the arrays of nodes $q$ lie in a compact set $K\subset\mathbb C$. Let
$$t_m=t_m(K)=\min_{q\in K^m}t(q)$$
be the optimal operator norm of $T_{F,q}$ as $q\in K^m$, and let
\[\varepsilon_m=\varepsilon_m(K)=\max_{q\in K^m}\varepsilon(q).\]

\par One of the main results of this paper is that
$$\lim_{m\to+\infty}m\,\varepsilon_m^{1/m}=\lim_{m\to+\infty}m\,t_m^{-1/m}=e\,d(K),$$
where $d(K)$ is the transfinite diameter of $K$. This is the content of Theorem \ref{T:mt2}.

\par By analogy with Chebyshev polynomials, we can think about the above limit as an {\it exponential capacity} of the compact set $K$. In Proposition \ref{P:dp} we show that not only the exponential capacity is a constant multiple of the transfinite diameter, but also the distribution of points maximizing $\varepsilon(q)$ and optimizing $t(q)$ is similar to the distribution of Fekete points. We refer to Section \ref{S:prelim} for the necessary definitions.

\par Again by analogy with approximating functions on the real line, we can look at the interpolation of functions whose ``spectrum" lies in $K$, i.e. functions which are Laplace transforms of complex measures $\mu$ on $K$. We show in Theorem \ref{T:Laplace1} that if
\[g(z)=L_F\mu(z)=\int\limits_KF(\zeta z)\,d\mu(\zeta)\]
and $Q_m$ is an $m$-tuple of Fekete points, then $\|T_{F,Q_m}g\|_\infty\leq|\mu|(K)$.

\par Finally, we look at the inverse problem in attempt to identify the spectrum of $f$ by the size of the coefficients of its interpolating $F$-polynomials. For this we introduce on $\mathcal F$
another norm,
$$\|f\|_1=\sum_{j=1}^m|c_j|,\;{\rm where}\;f(z)=\sum_{j=1}^kc_jF(q_jz),$$
and the $q_j$'s are distinct. Theorem \ref{T:PadeR} implies that if $\|T_{F,Q_m}g\|_1$ is bounded for some subsequence $\{m_k\}$ then the spectrum of $g$ lies in $K$. In the case $K=\Delta$ we show in Proposition \ref{P:circle} that the converse is true for a large set of measures on the unit circle. However, examples show that there are such measures $\mu$ for which the sequence $\{\|T_{F,Q_m}(L_F\mu)\|_1\}$ is unbounded when $Q_m$ is an $m$-tuple of Fekete points.

\section{Preliminaries}\label{S:prelim}
\par We recall here a few facts needed for the proofs of our theorems. 

\subsection{Vandermonde matrices and symmetric polynomials} 
\par We will need a formula for the inverse of the Vandermonde matrix 
$$A_m(q)=\begin{bmatrix}1&1&\dots&1\\q_1&q_2&\dots&q_m\\
\vdots&\vdots&\dots&\vdots\\q_1^{m-1}&q_2^{m-1}&\dots&q_m^{m-1}
\end{bmatrix},$$
where $q=(q_1,\dots,q_m)$. This formula can be found in \cite{MS58}. For the sake of the reader we include here a shorter proof kindly presented to us by Mark Kleiner.

\begin{Proposition}\label{P:iov} We have 
$$A^{-1}_m(q)=[v_{ik}]_{1\leq i,k\leq m}\;,\;\;v_{ik}=\frac{(-1)^{m-k}s_{m-k}(q_1,\dots,q_{i-1},q_{i+1},\dots,q_m)}
{\prod_{j=1,j\neq i}^m(q_i-q_j)}\;,$$
where $s_l(y_1,\dots,y_t)$ denotes the elementary symmetric polynomial of
degree $l$ in the variables $y_1,\dots,y_t$.
\end{Proposition}
\begin{proof} Let $q^i=(q_1,\dots,q_{i-1},q_{i+1},\dots,q_m)$. For every invertible matrix $A$ we have $A^{-1}=(\det A)^{-1}C^T$, where $C=[c_{ij}]$ is the cofactor matrix of $A$. Hence it suffices to
compute the cofactor $c_{ki}$ of the element $a_{ki}=q_i^{k-1}$ of the
matrix $A_m(q)$. Denote by $A^i_m(x)$ the matrix obtained by replacing
$q_i$ with the independent variable $x$ in the $i$-th column of the
Vandermonde matrix. By the cofactor expansion formula
\[\det A^i_m(x)=\sum_{k=1}^{m}c_{ki}x^{k-1}.\]
\par On the other hand $\det A^i_m(x)$ is a polynomial of degree $m-1$ in $x$
with roots in the set $q^i$ and with leading coefficient equal to $(-1)^{m+i}\det A_{m-1}(q^i)$. Hence
$$\det A^i_m(x)=(-1)^{m+i}\det A_{m-1}(q^i)\sum_{k=1}^{m}(-1)^{m-k}s_{m-k}(q^i)x^{k-1}.$$
Therefore $c_{ki}=(-1)^{i-k}s_{m-k}(q^i)\det A_{m-1}(q^i)$. Recalling that
\[\det A_m(q)=\prod_{1\le i<j\le m}(q_j-q_i)\] we obtain
\[v_{ik}=\frac{c_{ki}}{\det A_m(q)}=\frac{(-1)^{m-k}s_{m-k}(q^i)}{\prod_{j=1,j\neq i}^m(q_i-q_j)}\;.\]
\end{proof}

\medskip

\par We will also need to consider the determinants of the following generalized Vandermonde matrices:
\begin{equation}\label{e:Amk}
A_m^{j,k}(q)=\begin{bmatrix}1&1&\dots&1\\q_1&q_2&\dots&q_m\\ \vdots&\vdots&\dots&\vdots\\q_1^{j-1}&q_2^{j-1}&\dots&q_m^{j-1}\\q_1^{j+1}&q_2^{j+1}&\dots&q_m^{j+1}\\ \vdots&\vdots&\dots&\vdots\\ q_1^{m-1}&q_2^{m-1}&\dots&q_m^{m-1}\\q_1^k&q_2^k&\dots&q_m^k
\end{bmatrix},
\end{equation}
where $q=(q_1,\dots,q_m)$, $k>m-1>j\geq0$, or $k\geq m-1= j$.

\par The function $q\longrightarrow\det A_m^{j,k}(q)/\det A_m(q)$ is a symmetric polynomial of degree $k-j$. We refer for instance to \cite[Ch. 1]{Man} for the definitions and basic properties of symmetric polynomials and Schur polynomials. The above function is the Schur polynomial $s_\lambda$ corresponding to the partition $\lambda=(k-m+1,1,\dots,1,0,\dots,0)\in{\Bbb N}^m$ of length $l(\lambda)=m-j$.  By Littlewood's theorem (see \cite[Theorem 1.4.1]{Man}) it has the form 
\begin{equation}\label{e:jsp}
\frac{\det A_m^{j,k}(q)}{\det A_m(q)}=s_\lambda(q)=\sum_T q^{\mu(T)},
\end{equation}
where the sum runs over all semistandard (Young) tableaux $T$ with shape $\lambda$ numbered with positive integers less than or equal to $m$, and $\mu(T)=(\mu(T)_1,\dots,\mu(T)_m)$ is the weight of $T$. Note that 
$$\mu(T)_1+\ldots+\mu(T)_m=|\lambda|=(k-m+1)+1+\dots+1=k-j,$$ 
where $|\lambda|$ is the weight of $\lambda$. The number of terms of $s_\lambda$ is computed using the formula in \cite[Corollary 1.4.11]{Man} and is given by 
\begin{equation}\label{e:njsp}
s_\lambda(1,\dots,1)=\frac{k!}{j!\,(m-j-1)!\,(k-m)!\,(k-j)}\,.
\end{equation}
In particular, when $j=m-1$ this function is equal to the complete symmetric polynomial $h_{k-m+1}$ of degree $k-m+1$ in $m$ variables ,
\begin{equation}\label{e:csp}
\frac{\det A_m^{m-1,k}(q)}{\det A_m(q)}=h_{k-m+1}(q)=\sum_{i_1+\dots+i_m=k-m+1,i_j\geq0}q_1^{i_1}q_2^{i_2}\dots q_m^{i_m},
\end{equation}
where the above summation has $k\choose{m-1}$ terms (see also the Jacobi-Trudi formulas \cite[p.13]{Man}). 

\medskip

\subsection{Transfinite diameter}
\par Given a compact set $K\subset{\Bbb C}$ and $m\geq1$, let 
$$V_m=V_m(K)=\max_{q\in K^m}|\det A_m(q)|=\max_{q\in K^m}\prod_{1\leq j<i\leq m}|q_i-q_j|.$$
A collection $\{Q_m\in K^m:\,m\geq1\}$ is called an array of Fekete points if for each $m$ we have $V_m=|\det A_m(Q_m)|$. The sequence $V_m^{\frac{2}{m(m-1)}}$ decreases to a number $d(K)$, which is called the transfinite diameter of $K$ (see e.g. \cite[Ch. VII]{G}). 

\par An alternative characterization of the transfinite diameter $d(K)$ can be given using the Chebyshev polynomials of $K$. Let 
$$\tau_m=\tau_m(K)=\left(\min\{\|p\|_K:\,p(z)=z^m+c_1z^{m-1}+\dots+c_m,\,c_j\in{\Bbb C}\}\right)^{\frac{1}{m}}.$$
Then $\lim_{m\to+\infty}\tau_m=d(K)$ (see e.g. \cite[Ch. VII]{G}).

\par We recall from \cite[Ch. VII]{G} that 
\begin{equation}\label{e:tau}
\tau^{m-1}_{m-1}\le\frac{V_m}{V_{m-1}}\le m\,\tau^{m-1}_{m-1}.
\end{equation}
We also note that for every $q\in K^m$ we have 
\begin{equation}\label{e:gamma}
\gamma(q)^m\leq\prod_{i=1}^m\gamma_i(q)=|\det A_m(q)|^2\leq V_m^2.
\end{equation}

\section{Pad\'e interpolation by $F$-polynomials}\label{S:Pade}

\par Recall that given a holomorphic function $g$ in the disk $D_R$ and an $m$-tuple $q$ of distinct points, we defined $T_{F,q}g$ to be the unique $F$-polynomial with frequencies in $q$ which interpolates $g$ to order $m$ at $0$. We have the following: 

\begin{Theorem}\label{T:PadeR} Let $q=(q_1,\dots,q_m)$ be an $m$-tuple of distinct points of $\overline D_M$ and let $g$ be a holomorphic function in $D_R$. If $|z|<r<R$ then 
$$|g(z)-T_{F,q}g(z)|\leq\left(\frac{|z|}{r}\right)^m\left(\frac{r}{r-|z|}+\Gamma e^{2Mr}\right)\max\{|g(z)|:\,|z|=r\}.$$
\end{Theorem}

\begin{proof} Let us denote the derivatives of $g$ at $0$ by $g_n=g^{(n)}(0)$.

\par We start by deriving a formula for $T_{F,q}g$ in terms of fundamental Pad\'e interpolation $F$-polynomials. 
For $0\leq j\leq m-1$ let 
\begin{eqnarray*}
T_{F,q}^j(z)&=&\det \begin{bmatrix}1&1&\dots&1\\q_1&q_2&\dots&q_m\\ \vdots&\vdots&\dots&\vdots\\q_1^{j-1}&q_2^{j-1}&\dots&q_m^{j-1}\\q_1^{j+1}&q_2^{j+1}&\dots&q_m^{j+1}\\ \vdots&\vdots&\dots&\vdots\\ q_1^{m-1}&q_2^{m-1}&\dots&q_m^{m-1}\\F(q_1z)&F(q_2z)&\dots&F(q_mz)
\end{bmatrix}\\&=&(-1)^{m-j-1}\det A_m(q)\frac{F_j}{j!}\,z^j+O(z^m).
\end{eqnarray*}
Hence for $0\leq k\leq m-1$,
$$\left.\frac{d^kT_{F,q}^j}{dz^k}\right|_{z=0}=(-1)^{m-j-1}\det A_m(q)F_j\delta_{kj},$$
where $\delta_{kj}$ is the Kronecker delta. It follows that 
$$T_{F,q}g(z)=\sum_{j=0}^{m-1}g_j\,\frac{(-1)^{m-j-1}}{F_j\det A_m(q)}\,T_{F,q}^j(z).$$

\par Observe that 
$$g(z)-T_{F,q}g(z)=\frac{1}{\det A_m(q)}\,\det \begin{bmatrix}1&\dots&1&g_0/F_0\\q_1&\dots&q_m&g_1/F_1\\q_1^2&\dots&q_m^2&g_2/F_2\\\vdots&\vdots&\dots&\vdots\\ q_1^{m-1}&\dots&q_m^{m-1}&g_{m-1}/F_{m-1}\\F(q_1z)&\dots&F(q_mz)&g(z)\end{bmatrix}.$$
Indeed, this is seen by expanding the above determinant using the last column. 

\par Next, writing the functions in the last row in terms of their Taylor series at 0 we obtain:
\begin{eqnarray*}
g(z)-T_{F,q}g(z)&=&\sum_{k=m}^\infty\frac{F_kz^k}{\det A_m(q)k!}\,\det\begin{bmatrix}1&\dots&1&g_0/F_0\\q_1&\dots&q_m&g_1/F_1\\q_1^2&\dots&q_m^2&g_2/F_2\\\vdots&\vdots&\dots&\vdots\\ q_1^{m-1}&\dots&q_m^{m-1}&g_{m-1}/F_{m-1}\\q_1^k&\dots&q_m^k&g_k/F_k\end{bmatrix}\\
&=&\sum_{k=m}^\infty\frac{F_kz^k}{k!}\,\left(\sum_{j=0}^{m-1}(-1)^{m+j}\frac{g_j}{F_j}\,\frac{\det A_m^{j,k}(q)}{\det A_m(q)}+\frac{g_k}{F_k}\right),
\end{eqnarray*}
where $A_m^{j,k}(q)$ is the generalized Vandermonde matrix defined in (\ref{e:Amk}). 

\par Therefore we have 
$$g(z)-T_{F,q}g(z)=R_mg(z)+S_mg(z),$$
where 
$$R_mg(z)=\sum_{k=m}^\infty\frac{g_k}{k!}\,z^k$$
is the remainder of the Taylor series of $g$ at $0$, and 
$$S_mg(z)=\sum_{j=0}^{m-1}(-1)^{m+j}\frac{g_j}{F_j}\,\sum_{k=m}^\infty\frac{F_kz^k}{k!}\,\frac{\det A_m^{j,k}(q)}{\det A_m(q)}\;.$$

\par Let $C=\max\{|g(z)|:\,|z|=r\}$. By Cauchy's estimates we have 
$$|g_j|\leq j!\,Cr^{-j},\;j\geq0.$$
For $R_mg$ we use the well-known estimate
$$|R_mg(z)|\leq\sum_{k=m}^\infty\frac{|g_k|}{k!}\,|z|^k\leq C\sum_{k=m}^\infty\frac{|z|^k}{r^k}=C\left(\frac{|z|}{r}\right)^m\frac{r}{r-|z|}\;.$$

\par Next, we estimate $S_mg$. Since $|q_j|\leq M$ and $k\geq m$, we see by (\ref{e:jsp}) and (\ref{e:njsp}) that
$$\left|\frac{\det A_m^{j,k}(q)}{\det A_m(q)}\right|\leq\frac{k!\,M^{k-j}}{j!\,(m-j-1)!\,(k-m)!\,(k-j)}\leq\frac{k!\,M^{k-j}}{j!\,(m-j)!\,(k-m)!}\;.$$
Hence using (\ref{e:Fn}),
\begin{eqnarray*}
\left|\sum_{k=m}^\infty\frac{F_kz^k}{k!}\,\frac{\det A_m^{j,k}(q)}{\det A_m(q)}\right|&\leq&\Gamma\sum_{k=m}^\infty\frac{|z|^kM^{k-j}}{j!\,(m-j)!\,(k-m)!}\\
&=&\Gamma\,\frac{|z|^mM^{m-j}}{j!\,(m-j)!}\,\sum_{k=m}^\infty\frac{(M|z|)^{k-m}}{(k-m)!}=\Gamma e^{M|z|}\,\frac{|z|^mM^{m-j}}{j!\,(m-j)!}\;.
\end{eqnarray*}
Using (\ref{e:Fn}) again and the Cauchy estimates for $|g_j|$ we get
\begin{eqnarray*}
|S_mg(z)|&\leq&\Gamma e^{M|z|}|z|^m\sum_{j=0}^{m-1}\frac{|g_j|}{|F_j|}\,\frac{M^{m-j}}{j!\,(m-j)!}
\leq C\Gamma e^{M|z|}|z|^m\sum_{j=0}^{m-1}\frac{r^{-j}M^{m-j}}{(m-j)!}\\
&\leq& C\Gamma e^{M|z|}\left(\frac{|z|}{r}\right)^me^{Mr}\leq C\Gamma e^{2Mr}\left(\frac{|z|}{r}\right)^m.
\end{eqnarray*}
This concludes the proof of the theorem.
\end{proof}

\medskip

\par Next, we obtain sharp estimates for $\varepsilon(q)$ and $t(q)$.

\medskip

\begin{Theorem}\label{T:mt1}  Let $q=(q_1,\dots,q_m)$ be an $m$-tuple of distinct points of $\overline D_M$. Then 
$$t(q)\varepsilon(q)\geq1\;,\;\;\frac{(m-1)!}{\Gamma\gamma(q)}\leq t(q)\leq \frac{e^M(m-1)!}{\gamma(q)}\;,\;\;\frac{e^{-M}\gamma(q)}{(m-1)!}\leq\varepsilon(q)\le \frac{\Gamma e^M\gamma(q)}{(m-1)!}\,.$$
\end{Theorem}

\begin{proof} For the first inequality, let $f(z)=\sum_{j=1}^mc_jF(q_jz)$ with $\|f\|_\infty=1$. Since $f=T_{F,q}f$ we have $1=\|T_{F,q}f\|_\infty\leq t(q)\|f\|_\Delta$, which implies that $t(q)\varepsilon(q)\geq1$.

\medskip

\par We use the notations introduced in Section \ref{S:prelim} and prove now the estimates for $t(q)$. Let $g\in A(\Delta)$ with $\|g\|_\Delta\leq1$. Writing 
$$T_{F,q}g(z)=\sum_{j=1}^mc_jF(q_jz),$$ 
we have 
$$g_l:=g^{(l)}(0)=F_l\sum_{j=1}^mc_jq_j^l,\;0\leq l\leq m-1.$$
Hence
$$\begin{bmatrix}c_1\\ \vdots \\ c_m\end{bmatrix}=A_m(q)^{-1}
\begin{bmatrix}\frac{g_0}{F_0}\\ \vdots \\ \frac{g_{m-1}}{F_{m-1}}\end{bmatrix},\;{\rm or}\;\;
c_i=\sum_{k=1}^m v_{ik}\,\frac{g_{k-1}}{F_{k-1}}\;,$$
where $A_m(q)^{-1}=[v_{ik}]_{1\leq i,k\leq m}$. We use the formulas for $v_{ik}$ from Proposition \ref{P:iov}. Since $s_{m-k}(q_1,\dots,q_{i-1},q_{i+1},\dots,q_m)$ has $\binom{m-1}{m-k}$ terms, each a product of $m-k$ different $q_j$'s, it follows that
$$|v_{ik}|\leq\frac{1}{\gamma_i(q)}\,\binom{m-1}{m-k}\,M^{m-k}.$$
By Cauchy's estimates $|g_{k-1}|\leq (k-1)!$, and by (\ref{e:Fn}) $|F_{k-1}|\geq1$. Hence we obtain 
$$|c_i|\leq\frac{(m-1)!}{\gamma_i(q)}\,\sum_{k=1}^m\frac{M^{m-k}}{(m-k)!}
\leq\frac{(m-1)!}{\gamma(q)}\,e^M.$$
Therefore 
$$\|T_{F,q}g\|_\infty=\max_{1\leq i\leq m}|c_i|\leq\frac{e^M(m-1)!}{\gamma(q)}\,,$$
which yields the upper estimate on $t(q)$.

\par Observe that if $g(z)=z^{m-1}$ then, by above, $T_{F,q}g$ has coefficients 
$$c_i=v_{im}\,(m-1)!/F_{m-1}.$$ 
By Proposition \ref{P:iov}, $|v_{im}|=1/\gamma_i(q)$. Since $|F_n|\leq\Gamma$ we obtain 
$$t(q)\geq\|T_{F,q}g\|_\infty=\max_{1\leq i\leq m}|c_i|\geq\frac{(m-1)!}{\Gamma\gamma(q)}\;.$$

\medskip

\par The lower estimate on $\varepsilon(q)$ follows using $\varepsilon(q)\geq1/t(q)$ and the upper bound for $t(q)$. For the upper estimate on $\varepsilon(q)$ we consider the $F$-polynomial   
\begin{equation}\label{e:deff}
f(z)=\alpha\;\det\begin{bmatrix}1&1&\dots&1\\q_1&q_2&\dots&q_m\\
\vdots&\vdots&\dots&\vdots\\q_1^{m-2}&q_2^{m-2}&\dots&q_m^{m-2}\\F(q_1z)&F(q_2z)&\dots&F(q_mz)
\end{bmatrix},
\end{equation}
where $\alpha>0$ is chosen so that $\|f\|_\infty=1$. Note that $f$ vanishes to order $m-1$ at 0. 
Expanding the determinant using the last row we obtain
$$f(z)=\alpha\sum_{j=1}^m(-1)^{m+j}F(q_jz)\det A_{m-1}(q^j)=\sum_{j=1}^m c_jF(q_jz),$$
where 
$$q^j=(q_1,\dots,q_{j-1},q_{j+1},\dots,q_m),\;c_j=(-1)^{m+j}\alpha\det A_{m-1}(q^j).$$ 
Observe that for each $j$, $|\det A_m(q)|=|\det A_{m-1}(q^j)|\gamma_j(q)$, so 
$$\max_{1\leq j\leq m}|c_j|=\alpha\,\frac{|\det A_m(q)|}{\gamma(q)}\,.$$
Hence $\|f\|_\infty=1$ if we choose $\alpha=\gamma(q)/|\det A_m(q)|$. 

\par Using the Taylor series at 0 of $F(q_jz)$ it follows from (\ref{e:deff}) that 
$$f(z)=\frac{\gamma(q)}{|\det A_m(q)|}\sum_{k=m-1}^\infty\frac{F_k}{k!}\,(\det A_m^{m-1,k}(q))\,z^k,$$
where $A_m^{m-1,k}(q)$ is the generalized Vandermonde matrix given in (\ref{e:Amk}).

\par Equations (\ref{e:Fn}) and (\ref{e:csp}) now imply that 
\begin{eqnarray*}
\varepsilon(q)\leq\|f\|_\Delta&\leq&\Gamma\gamma(q)\sum_{k=m-1}^\infty\frac{|h_{k-m+1}(q)|}{k!}\\
&\leq&\Gamma\gamma(q)\sum_{k=m-1}^\infty\frac{1}{k!}{k\choose{m-1}}M^{k-m+1}
=\frac{\Gamma\gamma(q)}{(m-1)!}\,e^M.
\end{eqnarray*}
This concludes the proof of the theorem.
\end{proof}

\section{Exponential capacity}\label{S:EC}
\par Let $K\subset{\Bbb C}$ be a compact set. We now relate the optimal norm $t_m=t_m(K)$ of the Pad\'e interpolation operators $T_{F,q}$ to the quantities $\varepsilon_m=\varepsilon_m(K)$ and study their asymptotic growth as $m\to+\infty$. We use the notations introduced in Section \ref{S:prelim}. 

\medskip

\begin{Theorem}\label{T:mt2} If $M$ is the minimal radius of a disk centered at 0 and containing $K$ then $1\leq t_m\varepsilon_m\leq\Gamma e^{2M}$. Moreover, 
$$\lim_{m\to+\infty}m\,\varepsilon_m^{1/m}=\lim_{m\to+\infty}m\,t_m^{-1/m}=e\,d(K),$$ 
where $d(K)$ is the transfinite diameter of $K$.
\end{Theorem}

\begin{proof} By Theorem \ref{T:mt1} we have for every $q\in K^m$ that 
$$1\leq t(q)\varepsilon(q)\leq t(q)\varepsilon_m\;,\;\;\Gamma e^{2M}\geq t(q)\varepsilon(q)\geq t_m\varepsilon(q).$$
Hence $1\leq t_m\varepsilon_m\leq\Gamma e^{2M}$.

\par For the limits in the conclusion of the theorem, we will show that 
\begin{equation}\label{e:em}
\frac{e^{-M}}{(m-1)!}\,\tau_{m-1}^{m-1}\leq\varepsilon_m\leq\frac{\Gamma e^M}{(m-1)!}\,V_m^{\frac{2}{m}}\,.
\end{equation}
The theorem then follows by using Stirling's formula and since (see Section \ref{S:prelim})
$$\tau_m\to  d(K),\;V_m^{\frac{2}{m(m-1)}}\to  d(K).$$

\medskip

\par The upper estimate on $\varepsilon_m$ claimed in (\ref{e:em}) follows at once from Theorem \ref{T:mt1},  since for every $q\in K^m$ we have by (\ref{e:gamma}) that $\gamma(q)\leq V_m^{2/m}$.

\par For the lower estimate, let $q=(q_1,\dots,q_m)\in K^m$ be chosen so that $V_m=|\det A_m(q)|$ 
and let $q^1=(q_2,\dots,q_m)$. Then $V_m=\gamma_1(q)|\det A_{m-1}(q^1)|$ and, consequently, 
$\gamma_1(q)\ge V_m/V_{m-1}$. Hence by (\ref{e:tau}), $\gamma_1(q)\ge \tau^{m-1}_{m-1}$. Since by  symmetry the same inequality holds for all $\gamma_i(q)$ we see that $\gamma(q)\ge
\tau^{m-1}_{m-1}$. The lower estimate on $\varepsilon_m$ in (\ref{e:em}) now follows from Theorem \ref{T:mt1}. 
\end{proof}

\medskip

\par Next, we study the distribution of points maximizing $\varepsilon(q)$ and show that it is similar to the distribution of Fekete points.

\begin{Proposition}\label{P:dp} Consider an array of points $Q_m=(q_{1m},\dots,q_{mm})\in K^m$, where $m\geq1$. If $m\,\varepsilon(Q_m)^{1/m}\to  e\,d(K)$ then $|\det A_m(Q_m)|^{\frac2{m(m-1)}}\to  d(K)$. Moreover, in this case the sequence of measures $\mu_m=\frac{1}{m}\,\sum_{i=1}^m\delta_{q_{im}}$ converges weakly to the equilibrium measure of $K$.
\par Conversely, if $(|\det A_m(Q_m)|/V_m(K))^{1/m}\to 1$ then $m\,\varepsilon(Q_m)^{1/m}\to  e\,d(K)$.
\end{Proposition}

\begin{proof} Using Stirling's formula, it follows from Theorem \ref{T:mt1} that $m\,\varepsilon(Q_m)^{1/m}\to e\,d(K)$ if and only if $\gamma(Q_m)^{\frac{1}{m-1}}\to  d(K)$. 

\par By (\ref{e:gamma}) we have that 
$$\gamma(Q_m)^{\frac{1}{m-1}}\leq|\det A_m(Q_m)|^{\frac{2}{m(m-1)}}\leq V_m^{\frac{2}{m(m-1)}}.$$

\par If $\gamma(Q_m)^{\frac{1}{m-1}}\to  d(K)$ this implies that 
\[\lim_{m\to\infty}|\det A_m(Q_m)|^{\frac2{m(m-1)}}=d(K).\]
By \cite[Theorem 1.5]{BBCL} the measures $\mu_m$ converge weakly to the equilibrium measure of $K$.

\par Conversely, suppose that the sequence $(|\det A_m(Q_m)|/V_m)^{1/m}$
converges to 1. Since for each $i$, $|\det A_m(Q_m)|=\gamma_i(Q_m)|\det A_{m-1}(q_{1m},\dots,\widehat{q_{im}},\dots,q_{mm})|$, we see that
$$\gamma(Q_m)\ge\frac{|\det A_m(Q_m)|}{V_{m-1}}\;,\;\;{\rm so}\;\;\gamma(Q_m)^{\frac{1}{m-1}}\ge\left(\frac{|\det A_m(Q_m)|}{V_m}\right)^{\frac{1}{m-1}}
\left(\frac{V_m}{V_{m-1}}\right)^{\frac{1}{m-1}}\;.$$
Using (\ref{e:tau}) and (\ref{e:gamma}) we obtain
$$V_m^{\frac{2}{m(m-1)}}\geq \gamma(Q_m)^{\frac{1}{m-1}}\ge\tau_{m-1}\left(\frac{|\det A_m(Q_m)|}{V_m}\right)^{\frac{1}{m-1}},$$
which implies that $\lim_{m\to\infty}\gamma(Q_m)^{\frac{1}{m-1}}=d(K)$. 
\end{proof}

\medskip

\par We conclude this section by looking at the number of zeros of $F$-polynomials. When $F(z)=e^z$ general bounds on the number of zeros of exponential polynomials were obtained by Tijdeman \cite{T} (see also \cite{B}). Here we consider a ``minimal" number of zeros $N_m(K)$ defined in the same spirit as the quantity $\varepsilon_m(K)$. 

\par Given $q=(q_1,\dots,q_m)\in K^m$ so that the $q_j$'s are distinct we let $N(q)$ be the maximal number of zeros in $\Delta$ of the functions $f(z)=\sum_{j=1}^mc_jF(q_jz)$ with $f\not\equiv0$. Then we define 
$$N_m=N_m(K)=\min_{q\in K^m}N(q).$$

\begin{Proposition}\label{P:Nm} Assume $K\subset{\Bbb C}$ is a compact non-polar set. Then $N_m\geq m-1$ for every $m\geq1$, and $\lim_{m\to+\infty}\frac{N_m}{m\log m}=0$.
\end{Proposition}

\begin{proof} For any $q=(q_1,\dots,q_m)\in K^m$ one can construct, as in the proof of Theorem \ref{T:mt1}, a function $f(z)=\sum_{j=1}^mc_jF(q_jz)$ which vanishes to order $m-1$ at 0. This implies that $N_m\geq m-1$.

\par Fix now $q=(q_1,\dots,q_m)\in K^m$ with $\varepsilon(q)=\varepsilon_m$. Let $f(z)=\sum_{j=1}^mc_jF(q_jz)$ be so that $\|f\|_\Delta=1$. It follows from the definition of $\varepsilon(q)$ that 
$$\|f\|_\infty=\max_{1\leq j\leq m}|c_j|\leq \varepsilon(q)^{-1}=\varepsilon_m^{-1}.$$
If $n_f$ is the number of zeros of $f$ in $\Delta$ and $r>1$ is arbitrary one has that (see e.g. \cite[Theorem 2.2]{CP1} and its proof)
$$\left(\frac{r^2+1}{2r}\right)^{n_f}\leq\max_{|z|\leq r}|f(z)|.$$
Note that by (\ref{e:Fn}), $|F(q_jz)|\leq \Gamma e^{Mr}$, where $M$ is the minimal radius of a disk centered at 0 and containing $K$. Hence $\max_{|z|\leq r}|f(z)|\leq m\Gamma e^{Mr}/\varepsilon_m$. Since $f$ was arbitrary, we obtain using (\ref{e:em}) that 
$$\left(\frac{r^2+1}{2r}\right)^{N(q)}\leq\frac{m\Gamma e^{Mr}}{\varepsilon_m}\leq\frac{\Gamma e^{M(r+1)}m!}{\tau_{m-1}^{m-1}}\,.$$

\par As $N_m\leq N(q)$, it follows that 
$$N_m\,\log\frac{r^2+1}{2r}\leq\log\Gamma+M(r+1)+m\log m-(m-1)\log\tau_{m-1}.$$
Since $K$ is non-polar we have $d(K)>0$. Hence  
$$\log\frac{r^2+1}{2r}\,\limsup_{m\to+\infty}\frac{N_m}{m\log m}\leq1.$$
Letting $r\to\infty$ we conclude that $\lim_{m\to+\infty}\frac{N_m}{m\log m}=0$. 
\end{proof}

\section{Interpolation of Laplace transforms}\label{S:Laplace}

\par We have seen in Theorems \ref{T:PadeR} and \ref{T:mt1} that the Pad\'e-Taylor interpolators $T_{F,q}g$ of a holomorphic function $g$ in the disk $D_R$ converge locally uniformly to $g$ in $D_R$, but their norm $\|T_{F,q}g\|_\infty$ may be very large. We study now the situation when the interpolators have bounded norm. As we shall see, this problem is connected to the range of the Laplace transform. 

\par If $\mu$ is a complex measure on ${\Bbb C}$ with compact support, we define its Laplace transform $L_F\mu$ with respect to $F$ to be the entire function 
$$L_F\mu(z)=\int F(z\zeta)\,d\mu(\zeta)\,,\;z\in{\Bbb C}.$$
It is easy to see that an entire function $g=L_F\mu$ for some complex measure $\mu$ supported in a compact set $K$ if and only if there exists a sequence of $F$-polynomials with frequencies in $K$ which converges locally uniformly to $g$ and is bounded in the norm $\|\cdot\|_1$. 

\par Given an $m$-tuple of distinct points $Q_m=(q_{1m},\dots,q_{mm})$, we denote by 
$$l_i(Q_m,\zeta)=\frac{\prod_{j=1,j\neq i}^m(\zeta-q_{jm})}{\prod_{j=1,j\neq i}^m(q_{im}-q_{jm})},\;1\leq i\leq m,$$
the fundamental Lagrange interpolation polynomials with nodes in $Q_m$. 
 
\begin{Theorem}\label{T:Laplace1}
(i) Let $g$ be a holomorphic function in a disk $D_R$ and assume that there exist $m$-tuples of distinct points $Q_m=(q_{1m},\dots,q_{mm})\in K^m$ so that $\liminf_{m\to+\infty}\|T_{F,Q_m}g\|_1<+\infty$. Then there exists a complex measure $\mu$ supported on $K$ so that 
$g(z)=L_F\mu(z)$ for $z\in D_R$.
Moreover, $T_{F,Q_m}g$ converges locally uniformly on ${\Bbb C}$ to $L_F\mu$. 

\medskip

\par (ii) Conversely, let  $g=L_F\mu$, where $\mu$ is a complex measure supported on $K$, and let  $Q_m=(q_{1m},\dots,q_{mm})\in K^m$ be an $m$-tuple of distinct points. Then 
$$T_{F,Q_m}g(z)=\sum_{i=1}^m\left(\int_Kl_i(Q_m,\zeta)\,d\mu(\zeta)\right)F(q_{im}z)\,.$$

\par Moreover, if $Q_m$ is an $m$-tuple of Fekete points for $K$, then $\|T_{F,Q_m}g\|_\infty\leq|\mu|(K)$.
\end{Theorem}

\begin{proof} $(i)$ By passing to a subsequence, we may assume that $\{\|T_{F,Q_m}g\|_1\}_{m\geq1}$ is a bounded sequence. Let $g_m(z)=T_{F,Q_m}g(z)$. We write 
$$g_m(z)=\sum_{i=1}^mc_{im}F(q_{im}z)=L_F\mu_m(z),\;{\rm where}\;\mu_m=\sum_{i=1}^mc_{im}\delta_{im},$$
and $\delta_{im}$ denotes the Dirac mass at $q_{im}$. Since $\mu_m$ is supported on $K$ and $\|T_{F,Q_m}g\|_1=|\mu_m|(K)$ is a bounded sequence, there exists a subsequence $\mu_{m_k}$ which converges weakly to a complex measure $\mu$ supported on $K$. Hence for each $z\in{\Bbb C}$,
$$g_{m_k}(z)=L_F\mu_{m_k}(z)\to L_F\mu(z).$$
By Theorem \ref{T:PadeR} $g_m$ converges to $g$ locally uniformly on $D_R$, so $g=L_F\mu$ on $D_R$. As $L_F\mu$ is entire, Theorem \ref{T:PadeR} implies that $g_m\to L_F\mu$ locally uniformly on ${\Bbb C}$. 

\medskip

\par $(ii)$ Let us write $F(z)=F_{m-1}(z)+O(z^m)$, where $F_{m-1}$ is the Taylor polynomial of $F$ at $0$ of degree $m-1$. For each $z$, $F_{m-1}(z\zeta)$ is a polynomial of degree $m-1$ in $\zeta$, hence
$$F_{m-1}(z\zeta)=\sum_{i=1}^ml_i(Q_m,\zeta)F_{m-1}(q_{im}z).$$
It follows that 
\begin{eqnarray*}
g(z)&=&\int_KF_{m-1}(z\zeta)\,d\mu(\zeta)+O(z^m)\\
&=&\sum_{i=1}^m\left(\int_Kl_i(Q_m,\zeta)\,d\mu(\zeta)\right)F_{m-1}(q_{im}z)+O(z^m)\\
&=&\sum_{i=1}^m\left(\int_Kl_i(Q_m,\zeta)\,d\mu(\zeta)\right)F(q_{im}z)+O(z^m),
\end{eqnarray*}
which implies the formula for $T_{F,Q_m}g$.

\par Assume next that $Q_m$ is an $m$-tuple of Fekete points for $K$. Let $A^i_m(\zeta)$ be the matrix obtained by replacing $q_{im}$  with the independent variable $\zeta$ in the Vandermonde matrix $A_m(Q_m)$. By the definition of Fekete points, we have for each $i$ and $\zeta\in K$ that
$$|l_i(Q_m,\zeta)|=\left|\frac{\det A^i_m(\zeta)}{\det A_m(Q_m)}\right|\leq1,\;\;{\rm so}\; 
\left|\int_Kl_i(Q_m,\zeta)\,d\mu(\zeta)\right|\leq|\mu|(K).$$
This concludes the proof of the theorem. 
\end{proof}

\medskip

\par We consider now the range of the Laplace transform $L_F$, i.e. the class of entire functions $g$ of the form $g=L_F\mu$, where $\mu$ is a complex measure supported on $K$. We start with a simple remark:

\begin{Lemma}\label{L:balayage} For any complex measure $\mu$ supported on $K$ there exists o complex measure $\nu$ supported on the exterior boundary $\partial_eK$ of $K$ such that $L_F\mu=L_F\nu$ and $|\nu|(\partial_eK)\leq|\mu|(K)$.
\end{Lemma}

\begin{proof} Let $\widehat K$ be the polynomial hull of $K$, so $\partial_eK=\partial\widehat K$. Let $A(\partial_eK)\subset C(\partial_eK)$ be the subspace of complex-valued continuous functions on $\partial_eK$ which are uniform limits of polynomials on $\partial_eK$. Equivalently, $h\in A(\partial_eK)$ if and only if there exists a function $\widetilde h$ continuous on $\widehat K$ and holomorphic on its interior, such that $\widetilde h=h$ on $\partial_eK$. We define the linear functional 
$${\mathcal L}:A(\partial_eK)\longrightarrow{\Bbb C},\;\;{\mathcal L}(h)=\int_K\widetilde h\,d\mu\,.$$
It follows by the maximum principle that 
$$|{\mathcal L}(h)|\leq|\mu|(K)\,\|\widetilde h\|_K=|\mu|(K)\,\|h\|_{\partial_eK}\,.$$
The Hahn-Banach theorem implies that ${\mathcal L}$ extends to a bounded linear functional ${\mathcal L}$ on $C(\partial_eK)$ with $\|{\mathcal L}\|\leq|\mu|(K)$. By the Riesz representation theorem there exists a measure $\nu$ supported on $\partial_eK$ so that ${\mathcal L}(h)=\int h\,d\nu$ for every $h\in C(\partial_eK)$, and $|\nu|(\partial_eK)=\|{\mathcal L}\|\leq|\mu|(K)$. Hence for any function $h$ continuous on $\widehat K$ and holomorphic on its interior we have that 
$$\int_Kh\,d\mu=\int_{\partial_eK}h\,d\nu,$$
and the lemma follows.
\end{proof}

\par An interesting question is the following: given a complex measure $\mu$ on $K$, do there exist $m$-tuples of points $Q_m\in K^m$, $m\geq1$, so that the sequence $\{\|T_{F,Q_m}(L_F\mu)\|_1\}$ is bounded? By Theorem \ref{T:Laplace1}, this means that the sequence 
$$\Lambda(Q_m,\mu)=\sum_{j=1}^m\left|\int_Kl_j(Q_m,\zeta)\,d\mu(\zeta)\right|$$
is bounded. A reasonable choice is to take $Q_m$ to be an $m$-tuple of Fekete points. In the case of the unit disk $K=\Delta$ we show hereafter that Fekete points do not work for all measures. 

\par If $K=\Delta$ the $m$-tuples consisting of the roots of unity of order $m$ are Fekete points, 
$$Q_m=(\zeta_1,\dots,\zeta_m),\;\zeta_j=\exp(2\pi ij/m).$$
It is easy to see that in this case 
$$l_m(Q_m,\zeta)=\frac{1}{m}\,(1+\zeta+\dots+\zeta^{m-1})=\frac{\zeta^m-1}{m(\zeta-1)}\;,\;\;l_j(Q_m,\zeta)=l_m(Q_m,\zeta/\zeta_j)\,.$$
We consider measures $\mu$ supported on the unit circle ${\mathbf T}$ which are absolutely continuous with respect to the Lebesgue measure, so $d\mu=\chi(e^{i\theta})\,d\theta$ with $\chi\in L^1({\mathbf T})$. In this case we write $\Lambda(Q_m,\chi)$ for $\Lambda(Q_m,\mu)$.

\par Let $Abs({\mathbf T})$ denote the space of (continuous) functions $\chi$ on ${\mathbf T}$ whose Fourier series converges absolutely, i.e.
$$\sum_{n=-\infty}^\infty|\hat\chi(n)|<+\infty\,,\;\;\hat\chi(n)=\frac{1}{2\pi}\,\int_0^{2\pi}\chi(e^{i\theta})e^{-in\theta}\,d\theta\,.$$
By a theorem of Bernstein, this space contains the class of H\"older continuous functions with H\"older exponent greater than $1/2$, while by a theorem of Zygmund it contains the class of all H\"older continuous functions with bounded variation (see e.g. \cite[Chapter I]{Kat}).

\begin{Proposition}\label{P:circle} If $Q_m$ are as above and $\chi\in Abs({\mathbf T})$ then $\displaystyle\sup_{m\geq1}\Lambda(Q_m,\chi)<+\infty$. However, one has that $\displaystyle\sup_{m\geq1}\Lambda(Q_m,\chi)=+\infty$ for $\chi$ in a dense $G_\delta$ subset of $L^1({\mathbf T})$.
\end{Proposition}

\begin{proof} Suppose first that $\chi\in Abs({\mathbf T})$. Then 
$$\chi(e^{i\theta})=\sum_{n=-\infty}^\infty\hat\chi(n)e^{in\theta},$$
and the Fourier series converges uniformly to $\chi$ on ${\mathbf T}$. We define 
$$\chi^-(e^{i\theta})=\sum_{n=0}^\infty\hat\chi(-n)e^{-in\theta},\;\chi_m^-(e^{i\theta})=\sum_{n=0}^{m-1}\hat\chi(-n)e^{-in\theta}.$$
Then $\chi^-$ is continuous and $\chi_m^-$ converges uniformly to $\chi^-$ on ${\mathbf T}$. It follows that 
\begin{eqnarray*}
\int_0^{2\pi}l_j(Q_m,e^{i\theta})\chi(e^{i\theta})\,d\theta&=&\frac{1}{m}\,\int_0^{2\pi}\left(\sum_{k=0}^{m-1}\frac{e^{ik\theta}}{\zeta_j^k}\right)\left(\sum_{n=-\infty}^\infty\hat\chi(n)e^{in\theta}\right)d\theta\\
&=&\frac{2\pi}{m}\,\sum_{k=0}^{m-1}\hat\chi(-k)\zeta_j^{-k}=\frac{2\pi}{m}\,\chi_m^-(\zeta_j)\,,
\end{eqnarray*}
hence
$$\left|\Lambda(Q_m,\chi)-\frac{2\pi}{m}\,\sum_{j=1}^m|\chi^-(\zeta_j)|\right|\leq\frac{2\pi}{m}\,\sum_{j=1}^m|\chi_m^-(\zeta_j)-\chi^-(\zeta_j)|\leq2\pi\|\chi_m^--\chi^-\|_{\mathbf T},$$
where $\|\cdot\|_{\mathbf T}$ denotes the uniform norm on ${\mathbf T}$. Therefore the limit exists,
\begin{equation}\label{e:limLa}
\lim_{m\to\infty}\Lambda(Q_m,\chi)=\int_0^{2\pi}|\chi^-(e^{i\theta})|\,d\theta\,.
\end{equation}
This proves the first claim of the proposition.

\par For the second claim, we consider the family of seminorms $\Lambda(Q_m,\cdot)$ on the Banach space $L^1({\mathbf T})$. By Theorem \ref{T:Laplace1} they are bounded, i.e.
$$\Lambda(Q_m,\chi)\leq m\int_0^{2\pi}|\chi(e^{i\theta})|\,d\theta\,.$$
The Banach-Steinhaus theorem (which holds in this setting with the same proof, see e.g. \cite{Rud}) implies that either there exists a constant $C$ so that 
$$\Lambda(Q_m,\chi)\leq C\int_0^{2\pi}|\chi(e^{i\theta})|\,d\theta$$
holds for all $m$ and $\chi\in L^1({\mathbf T})$, or one has that $\displaystyle\sup_{m\geq1}\Lambda(Q_m,\chi)=+\infty$ for $\chi$ in a dense $G_\delta$ subset of $L^1({\mathbf T})$. Assume for a contradiction that the first case holds. By (\ref{e:limLa}), this would imply letting $m\to\infty$ that 
$$\int_0^{2\pi}|\chi^-(e^{i\theta})|\,d\theta\leq C\int_0^{2\pi}|\chi(e^{i\theta})|\,d\theta\,,$$
for every $\chi\in Abs({\mathbf T})$. This inequality is false, as shown by the functions 
$$\chi_r(e^{i\theta})=1+\sum_{n=-\infty}^\infty r^{|n|}e^{in\theta}=2\,\re\frac{1}{1-re^{i\theta}}\,,\;0<r<1.$$
Then $\chi_r^-(e^{i\theta})=1+1/(1-re^{-i\theta})$ and one checks that 
$$\lim_{r\to1}\int_0^{2\pi}|\chi_r^-(e^{i\theta})|\,d\theta=+\infty,\;{\rm while}\;\int_0^{2\pi}|\chi_r(e^{i\theta})|\,d\theta=O(1)\;{\rm as}\;r\to1.$$
This finishes the proof.
\end{proof}

\par In the case of the Dirac mass at a point of ${\mathbf T}$ we have the following estimates.

\begin{Example} If $Q_m$ are as above and $\mu$ is the Dirac mass at $-1$ then $\Lambda(Q_{2k},\mu)=1$, while $\lim_{k\to+\infty}\Lambda(Q_{2k+1},\mu)=+\infty$. Indeed, $\int l_j(Q_m,\zeta)\,d\mu(\zeta)=l_j(Q_m,-1)$. If $m=2k$ then $l_k(Q_{2k},-1)=1$, $l_j(Q_{2k},-1)=0$ for $j\neq k$, so $\Lambda(Q_{2k},\mu)=1$. Assume now $m=2k+1$. Then 
$$l_j(Q_{2k+1},-1)=\frac{2\zeta_j}{(2k+1)(\zeta_j+1)}\,,\;j=1,\dots,2k+1.$$
There exists a constant $C$ so that for all $k$ and all $j=1,\dots,2k+1$ we have 
$$|\zeta_j+1|=\left|\exp\left(2\pi i\left(\frac{j}{2k+1}-\frac{1}{2}\right)\right)-1\right|\leq C\left|\frac{j}{2k+1}-\frac{1}{2}\right|.$$
Thus 
$$\Lambda(Q_{2k+1},\mu)\geq\frac{4}{C}\,\sum_{j=1}^{2k+1}\frac{1}{|2j-2k-1|}\geq\frac{4}{C}\,\sum_{s=0}^k\frac{1}{2s+1}\,.$$
On the other hand, we note that if we replace $Q_m$ with the $m$-tuples $Q'_m$ of roots of order $m$ of $-1$ then $\Lambda(Q'_m,\mu)=1$ for all $m$. 
\end{Example}


\begin{thebibliography}{XXXXX}

\bibitem[B]{B} A. Baker, {\em Transcendental Number Theory}, Cambridge Univ. Press, 1975.

\bibitem[BBCL]{BBCL} T. Bloom, L. Bos, C. Christensen and N. Levenberg, {\em Polynomial interpolation of holomorphic functions in $\mathbb C$ and $\Bbb C^n$}, Rocky Mountain J. Math. {\bf 22} (1992), 441--470.

\bibitem[CP1]{CP1} D. Coman and E. A. Poletsky, {\em Measures of transcendency for entire functions}, Mich. Math. J. {\bf 51} (2003), 575--591.

\bibitem[CP2]{CP2} D. Coman and E. A. Poletsky, {\em Transcendence measures and algebraic growth of entire functions},  Invent. Math.  {\bf 170} (2007),  103--145.

\bibitem[CP3]{CP3} D. Coman and E. A. Poletsky, {\em Polynomial estimates, exponential curves and Diophantine approximation},  Math. Res. Lett.  {\bf 17} (2010),  1125--1136.

\bibitem[Ge]{Ge} A. O. Gelfond, {\em The problem of the representation and uniqueness of an entire analytic function of first order}, Usp. Mat. Nauk {\bf 3} (1937), 144--174.

\bibitem[Go]{G} G. M. Goluzin, {\em Geometric Theory of Functions of a Complex Variable}, Translations of Mathematical Monographs, {\bf 26}, American Mathematical Society, Providence, R.I., 1969.

\bibitem[HNP]{HNP} S. V. Hruschev, N. K. Nikolskii and B. S. Pavlov, {\em Unconditional bases of exponentials and of reproducing kernels}, Complex analysis and spectral theory, 214--335, Lecture Notes in Math. 864, Springer, Berlin-New York, 1981.

\bibitem[K]{Kat} Y. Katznelson, {\em An Introduction to Harmonic Analysis}, Dover Publications, Inc., 1976.

\bibitem[L]{Le} B. Ya. Levin, {\em Distribution of Zeros of Entire Functions}, Translations of Mathematical Monographs, {\bf 5}, American Mathematical Society, Providence, R.I., 1964.

\bibitem[MS]{MS58} N. Macon and A. Spitzbart, {\em Inverses of Vandermonde matrices}, Amer. Math. Monthly {\bf 65} (1958), 95--100.

\bibitem[Ma]{Man} L. Manivel, {\em Symmetric Functions, Schubert Polynomials and Degeneracy Loci} SMF/AMS Texts and Monographs, Volume 6, 2001.

\bibitem[R]{Rud} W. Rudin, {\em Real and Complex Analysis}, 3rd ed., McGraw-Hill, 1987.

\bibitem[T]{T} R. Tijdeman, {\em On the number of zeros of general exponential polynomials}, Indag. Math. {\bf 37} (1971), 1--7.

\end{thebibliography}
\end{document}